\documentclass[11pt]{article} 
 
\usepackage{latexsym} 
\usepackage{amsfonts} 
\usepackage{amsmath} 
\usepackage{amsthm} 
\usepackage{mathrsfs} 
\usepackage{dsfont} 
\usepackage{bbold} 
\usepackage[english]{babel} 
\usepackage{caption} 
\usepackage{epsfig} 
\usepackage{float} 
\usepackage{subfigure} 
\usepackage{psfrag} 
\usepackage{graphicx} 
\usepackage{epsfig} 
\usepackage{amsbsy} 
\usepackage{hyperref,color}

\DeclareMathOperator{\meas}{meas} 
\DeclareMathOperator{\dist}{dist}
 
\newtheorem{theorem}{Theorem} 
\newtheorem*{prop*}{Theorem} 
\newtheorem{theo}[theorem]{Theorem} 
 
\newtheorem{defi}[theorem]{Definition} 
\newtheorem{lemma}[theorem]{Lemma} 
 
\newtheorem{rmk}[theorem]{Remark}

\newcommand{\zerarcounters}{\setcounter{equation}{0}\setcounter{theorem}{0}} 
 
\newcommand{\pa}[1]{{\left(#1\right)}}

\newcommand{\norm}[1]{{\left \|#1\right \|}}
 
\newcommand{\ZZZ}{\mathds{Z}} 
\newcommand{\CCC}{\mathds{C}} 
\newcommand{\NNN}{\mathds{N}} 
 
\newcommand{\RRR}{\mathds{R}} 
\newcommand{\TTT}{\mathds{T}} 
\newcommand{\uno}{\mathds{1}} 
 
\newcommand{\CCCC}{{\mathcal C}} 
\newcommand{\DD}{{\mathcal D}}

\newcommand{\calH}{{\mathcal H}}

\newcommand{\LL}{{\mathcal L}} 
\newcommand{\MM}{{\mathcal M}} 
\newcommand{\NN}{{\mathcal N}}

\newcommand{\RR}{{\mathcal R}} 
 
\newcommand{\TT}{{\mathcal T}}

\newcommand{\matF}{{\mathscr F}} 
\newcommand{\matG}{{\mathscr G}} 
\newcommand{\matH}{{\mathscr H}}

\newcommand{\matP}{{\mathscr P}}

\newcommand{\matT}{{\mathscr T}} 
\newcommand{\matU}{{\mathscr U}}

\newcommand{\und}{\underline}

\newcommand{\ttV}{\mathtt{V}}
\newcommand{\ttW}{\mathtt{W}}

\newcommand{\Fullbox}{{\rule{2.0mm}{2.0mm}}} 
 
\newcommand{\EP}{\hfill\Fullbox\vspace{0.2cm}} 
\newcommand{\prova}{\noindent{\it Proof. }} 
\newcommand{\io}{\infty} 
\newcommand{\e}{\varepsilon} 

\newcommand{\al}{\alpha} 
\newcommand{\de}{\delta}

\newcommand{\m}{\mu}

\newcommand{\g}{\gamma} 
\newcommand{\om}{\omega}

\newcommand{\la}{\lambda} 
\newcommand{\f}{\varphi} 
\newcommand{\s}{\sigma} 
 
\newcommand{\del}{\partial}

\newcommand{\av}[1]{\langle #1 \rangle}

\newcommand{\ii}{{\rm i}}

\newcommand{\set}[1]{{\left\{#1\right\}}}

\newcommand{\jap}[1]{\langle #1 \rangle}

\def\ins#1#2#3{\vbox to0pt{\kern-#2 \hbox{\kern#1 #3}\vss}\nointerlineskip} 
 
\begin{document}
 
\title{\bf Long time stability for Lagrangian tori in infinite dimensional lattice systems}
 
\author{\textbf{Livia Corsi, Michela Procesi}\\ 
\small Dipartimento di Matematica e Fisica, Universit\`a Roma Tre, Roma, 
00146, Italy\\
\footnotesize e-mail: livia.corsi@uniroma3.it, michela.procesi@uniroma3.it}

\date{} 
 
\maketitle

\begin{abstract} 
We consider a mechanical system of infinitely many rotators weakly interacting with each other and prove that
the infinite dimensional invariant tori provided by \cite{CGP} are Lagrangian and stable over
subexponentially long time.

\smallskip

\end{abstract} 

 \tableofcontents

\zerarcounters 
\section{Introduction} 
\label{intro} 

In the theory of dynamical system it is natural to ask whether invariant manifolds
exist, are stable in some sense and, if yes, over what time. In particular, in the case of finite dimensional
close-to-integrable Hamiltonian systems it is well known that most Lagrangian tori survive any
perturbation small
enough and they are stable for exponentially long time. However such results do not
translate directly to the infinite dimensional setting, since the smallness assumption on the
perturbation as well as the stability time depend heavily on the dimension and do not pass
to the limit. For instance in \cite{Wa} is proved that the persistence of invariant tori for a $d$-lattice
with short range interaction
is guaranteed provided the size of the perturbation goes to zero as $d$ goes to infinity with
a quantitative control on such a dependence. This is coherent with the results of \cite{boka},
where the authors prove
the absence of long time stability for size-$\e$ hamiltonian perturbations of quasi-convex 
integrable systems, when the dimension $2d$ of the phase space becomes of order $\log(1/\e)$:
in fact here they construct diffusive orbits.

On the other hand such results may be bypassed by requiring that the perturbation has some special
structure, as seen for instance in \cite{BFG,po}.

 Very recently it has been proved in \cite{CGP} the existence a positive measure
set of invariant tori for a special class of close-to-integrable Hamiltonian system on the lattice $\ZZZ$.
Roughly, the special class consists of Hamiltonians where the integrable part is quadratic w.r.t.
the actions with some appropriate ``twist'' condition, while the perturbation depends on the angles
alone with some quantitative and rather strong analyticity condition:
this model encodes all the key difficulties of KAM theory, as it has been shown already in finite
dimension in \cite{gal}.

The idea in \cite{CGP} can be described as follows. If one tries to prove the 
existence of infinite dimensional invariant tori by constructing a convergent sequence of
finite dimensional tori whose dimension increases, one is stuck with the problem that the 
Diophantine condition gets worse and worse as the dimension increases; this may be
overcome by requiring stronger and stronger analyticity condition on the perturbation: the more
angles are involved, the stronger the analyticity required.
On the other hand in \cite{CGP} it is shown that one may work directly in the infinite dimensional
context by using the Bourgain infinite dimensional Diophantine vectors and encoding the
strong analyticity w.r.t.~the angle with an appropriate functional setting: once such a setting
has been established, the technique in \cite{gal} applies essentially unchanged.
Being a more direct approach, the technique used in \cite{CGP} allows to prove that
the set of invariant tori has relatively large probability measure.

The aim of the present paper is to show that the tori of \cite{CGP} are in fact stable for
subexponentially long time; see Theorem \ref{stimadeitempivera}.

Our proof is made of two main steps.

The first step consists in showing that the invariant tori constructed in \cite{CGP}
are lagrangian and admit adapted symplectic coordinates. More precisely we show that
there exists a symplectic change of variables so that the invariant torus,
and in fact the Hamiltonian, has a very simple expression; see below for details.
 The main idea to obtain
such a result is to generalize to the infinite dimensional context the result in \cite{h}.
This has been done in \cite{BB} for finite dimensional isotropic invariant tori
for PDEs. The main difference here is that our tori are infinite dimensional, and in fact
lagrangian instead of simply isotropic: however we take strong advantage from the fact
that our phase space is the product of an infinite dimensional
torus and a sequence space; see \eqref{spazio}.
Note that at the formal level such a result can be obtained as in the finite dimensional case;
the delicate part is to have quantitative bounds so that the change of variables is well posed
and preserves the regularity of the Hamiltonian.

The second step is to reformulate a standard non-resonant averaging argument (see for instance
\cite{giorgilli}) to the infinite dimensional context: here we take strong advantage of the 
bounds on the small divisors and ultraviolet cutoffs proved in \cite{MP,CMP}. The stability time increases
to infinity as either the size of the perturbation or the distance from the invariant torus goes to zero.
As customary, we first prove that we may perform $N$ averaging steps in a neighbourhood of the invariant 
torus whose size goes to zero as $N$ goes to infinity; then we optimize over $N$ and obtain a 
subexponential stability time. 

It is worth mentioning that, compared with the $2d$-dimensional case where the stability time is 
$\sim e^{(1/\e)^{1/(2d-1)}}$, $\e$ being the size of the perturbation, here we only obtain
subexponential time $\sim e^{(\log(1/\e))^{4/3}}$ due to the much worse small divisors.

Note that
a similar estimate is obtained in \cite{FG,BMP1} for the stability, in the Gevrey norm, of the elliptic fixed point at
zero of the NLS equation with external parameters. In the context of Hamiltonian PDEs, the closest
result to the one in the present paper is in \cite{Y}, where the authors prove polynomial stability
for the infinite dimensional KAM tori provided in \cite{Bjfa}.

\zerarcounters 
\section{Setting and main result} 
\label{setup} 

Consider the set 
\begin{equation}\label{spazio}
\matP:=\TTT^\ZZZ \times \ell_q^\io,
\end{equation}
where 
\[
\ell^\io_q=\ell^\io_q(\RRR):=\{ I\in \ell^2(\RRR)\,:\, \|I\|_q:=\sup_{j\in\ZZZ}\jap{j}^q|I_j| <\io \}
\]
is a Banach space based on $\ell^\io(\RRR)$ with $\jap{j}:=\max\{1,|j|\}$ and $q>1$,
while the torus $\TTT^\ZZZ$ is endowed with the metric
\[
\dist(\theta,\theta') := \sup_{j\in\ZZZ}\inf_{k\in\ZZZ}|\theta_j - \theta_j' - 2k\pi|,
\]
which makes it a Banach manifold based on $\ell^\io(\RRR)$.
Note that at each point $p\in\matP$ the tangent space is given by
\begin{equation}\label{tipi}
\TT:=T_p\matP= \ell^\io(\RRR) \times \ell^\io_q(\RRR),
\end{equation}
endowed with the norm
\[
\|X\|:= \max\{\|X^{(\theta)}\|_\io,\|X^{(I)}\|_q\},
\]
and we write
\begin{equation}\label{fibrato}
T\matP:=\bigcup_{p\in\matP} T_p\matP.
\end{equation}

\begin{defi}\label{domini}
	Given $I_0\in\ell_q^\io$, we set
	\[
	\matP_{r}(I_0):=\TTT^\ZZZ\times B_r(I_0;\ell^\io_q)\,.
	\]
	where 
	\[
		B_r(I_0;\ell^\io_q):= \{ I \in \ell^\io_q \, : \, \|I-I_0\|_q <r\}
	\]
It shall be conventient to denote $B_r(\ell^\io_q):=B_r(0;\ell^\io_q)$ and similarly $\matP_{r}=\matP_{r}(0)$.
\end{defi}

Set also
\begin{equation}\label{fin}
\begin{aligned}
\ZZZ^\ZZZ_f&:= \{ \nu\in\ZZZ^\ZZZ \, : \, \|\nu\|_{\ell^1}<\io\},\\
\NNN^\ZZZ_f&:= \{ k\in\NNN^\ZZZ \, : \, \|k\|_{\ell^1}<\io\},
\end{aligned}\qquad\qquad
|\nu|_{\star}:= \sum_{j\in\ZZZ}\sqrt{\jap{j}}|\nu_j|.
\end{equation}
%

\begin{defi}\label{bestie}
	Given any Banach space $(\ttV,\|\cdot\|_{\ttV})$ and any $r,s>0$,
	we say that a function $f:\matP_r \to \ttV$ is $(r,s)$-analytic if
	it represented by a totally convergent  Lindstedt series
	\begin{equation}\label{formal}
	f(\theta,I) = \sum_{k\in\NNN^\ZZZ_f}\sum_{\nu\in\ZZZ^\ZZZ_f} f_{k,\nu} I^ke^{\ii \theta\cdot \nu},\qquad\qquad 
		I^k :=\prod_{j\in\ZZZ}I_j^{k_j}
	\end{equation}
	with the condition
	\begin{equation}\label{converge!}
		\|f\|_{r,s,\ttV} := \sum_{k\in\NNN^\ZZZ_f}\sum_{\nu\in\ZZZ^\ZZZ_f} \|f_{k,\nu}\|_{\ttV} e^{s|\nu|_\star}(r w)^k <\io
	\end{equation}
	where $w:=\{\jap{j}^{-q}\}_{j\in\ZZZ}$. 
	If $f$ depends on the angles only we may say that it is $s$-analytic,
	whereas if it depends on the actions only we may say that it is $r$-analytic: in such a case
	we drop the extra index in the norm. 
Finally, given $I_0\in \ell^\infty_q$, setting $f_{I_0}(\theta,I) := f(\theta,I_0+I)$ we say that $f$ is $(r,s)$-analytic 
at $I_0$ if $f_{I_0}$ is $(r,s)$-analytic.
\end{defi}

\begin{rmk}\label{rappresento}
Note that since $\ell^\io_q$ is not separable, Definition \ref{bestie} is in fact quite stronger 
then classical analyticity. In fact, not only there exist continuous polynomials in $I$
that can be written as in \eqref{formal} but \eqref{converge!} does not
hold, but also there may be continuous polynomials that may not be represented as in \eqref{formal}.
\end{rmk}

\begin{rmk}\label{majo}
	Note that  if a formal series as in \eqref{formal} satisfies \eqref{converge!} then  the series is totally
	convergent, and it defines an holomorphic  function  $\matP^\CCC_{r,s}\to \ttV\times \ttV$,
	where
	\[
	\begin{aligned}
	&\matP^\CCC_{r,s}:=\TTT^\ZZZ_s\times B_r(\ell^\io_q(\CCC)), \\
	&\TTT^{\ZZZ}_s:=\{z\in(\CCC/2\pi\ZZZ)^\ZZZ \,:\, {\rm Re}( z) \in \TTT^\ZZZ\,,
	 |{\rm Im} (z_j)|<s \sqrt{\jap{j}} \,,  \} 
	\end{aligned}
	\]
	Similarly, if $f$ 
	$(r,s)$-analytic at $I_0$, then $f_{I_0}$ is $(r,s)$-analyitc, and  \eqref{formal} coincides with its
	Taylor-Fourier series about $I_0$.
\end{rmk}

\begin{defi}\label{differenziali}
Let $(\ttV,\|\cdot\|_\ttV)$ ans $(\ttW,\|\cdot\|_\ttW)$ be sequence Banach spaces and
let $\LL(\ttW,\ttV)$ be the set of continuous linear operators from $\ttW$ to $\ttV$ which admit
a matrix representation as
\[
({L}[w])_j := \sum_{k\in I_\ttW,}(L_{k}^j w_{k}),\qquad\qquad \forall \ j\in I_{\ttV}.
\]
We define
\begin{equation}\label{maj}
\MM(\ttW,\ttV):=\{L\in \LL(\ttW,\ttV) \;: \; \und{L} \in \LL(\ttW,\ttV)\},
\end{equation}
where $\und{L}$ is the Cauchy majorant defined component-wise as
\begin{equation}\label{maggiorante}
(\und{L}[w])_j := \sum_{k\in I_\ttW,}(|L_{k}^j|w_{k}),\qquad\qquad \forall \ j\in I_{\ttV},
\end{equation}
with $I_{\ttV}, I_{\ttW}$ the indices sets of $\ttV,\ttW$ respectively.
Note that $\MM(\ttW,\ttV)$ is a Banach space endowed with the operator norm 
\begin{equation}\label{normaop}
\| L\|_{\MM(\ttW,\ttV)}:= \sup_{\|w\|_{\ttW}=1}\|\und{L}[w]\|_\ttV.
\end{equation}
As customary,
if $\ttW=\ttV$ we shorten $\MM(\ttV,\ttV)=\MM(\ttV)$. If $\ttW=\TT$
andf $\ttV=\ell^\io_q$ we write $\| L\|_{\MM,q}=\| L\|_{\MM(\TT,\ttV)}$.
\end{defi}

\begin{defi}\label{campovett}
We say that a map $X:\matP_r \to T\matP$ is an $(r,s)$-analytic vector field if 
$$\|X\|_{r,s,T\matP}:= \frac{\|X^{(\theta)}\|_{r,s,\ell^\io}}{s}+ \frac{\|X^{(I)}\|_{r,s,\ell^\io_q}}{r}<\io .$$
\end{defi}


Recall that if $\Phi :\matP_{r'} \to \matP_{r}$ is a diffeomorphism and $\matF:\matP_r\to \ttV$
one defines the pull-back via $\Phi$ of $\matF$ as
\begin{equation}\label{pb}
\Phi^*\matF := \matF\circ \Phi.
\end{equation}

\begin{rmk}\label{commutatore}
Given $X,Y$ $(r,s)$-analytic vector fields, we may define their commutator $[X,Y]$ component-wise
in the standard way. Then the Cauchy
estimates ensure that $[X,Y]$ is $(r',s')$-analytic for any $r'<r$, $s'<s$.
\end{rmk}

Given an $(r,s)$-analytic Hamiltonian $H:\matP_r\to\RRR$, we define the corresponding Hamiltonian vector field component-wise as
\begin{equation}\label{vectorfield}
\begin{aligned}
X_H^{(\theta_j)} &= \del_{I_j} H \\
X_H^{(I_j)} &= - \del_{\theta_j} H,
\end{aligned}
\end{equation}
and note that $X_H$ can be written as a Lindsted series in the sense that each component can be represented 
as in \eqref{formal}, and \eqref{converge!} holds with $\ttV=\RRR$.

\begin{defi}\label{reghami}
We say that a Hamiltonian $H:\matP_r\to\RRR$ is $(r,s)-$\emph{regular} if it is analytic and the Hamiltonian vector field
$X_H:\matP_r\to T\matP$ is $(r,s)$-analytic as well. For compactness of notation we write
\[
\| H\|_{r,s}:= \|X_H\|_{r,s,T\matP}
\]
and denote the set of $(r,s)$-regular Hamiltomians as $\matH_{r,s}$.
\end{defi}

Note that if $H$ is a $(r,s)$-regular Hamiltonian, then $X_H$ is the unique vector field such that
\begin{equation}\label{unicop}
dH[Y] = \Omega[ X_H, Y] ,\qquad \forall\, Y \in T\matP,
\end{equation}
where $\Omega$ is the symplectic $2$-form defined as follows.
For any $p\in\matP$, we define on $T_p\matP$ a bilinear form $\Omega_p$ by setting, for any $X,Y\in T_p\matP$,
\begin{equation}\label{simpl}
\begin{aligned}
\Omega_p[X,Y]&:=dI\wedge d\theta [X,Y] \\
&:=\sum_{j\in\ZZZ} dI_j\wedge d\theta_j [(X_\theta,X_I),(Y_\theta,Y_I)] = \sum_{j\in\ZZZ} (X_{\theta_j}Y_{I_j} - X_{I_j}Y_{\theta_j})
\end{aligned}
\end{equation}
which is well defined, continuous, antisymmetric and non degenerate.
The map $\Omega_p$ does not depend on the point $p$ and hence it is trivially smooth: in particular
$\Omega$ acts on $C^1$-vector fields. Moreover
\[
\begin{aligned}
T\matP &\to (T^*\matP) \\
X &\mapsto \Omega[X,\cdot]
\end{aligned}
\]
is injective but not surjective. This makes $\Omega$ a weak symplectic structure on $\matP$.
As in \cite{BMPduca}, we introduce the following product measure.

\begin{defi} \label{probmeas}
We endow $B_r(\ell^\io_q)$ with the product probability measure $\m$ induced
from the product measure on $[-1/2,1/2]^\ZZZ$ by the map
\[
\begin{aligned}
&B_r(\ell^\io_q) \longrightarrow [-1/2,1/2]^\ZZZ \\
&\{I_{j}\}_{j\in\ZZZ} \longmapsto \{\jap{j}^q I_j/(2r)\}.
\end{aligned}
\]
\end{defi}

Following \cite{Bjfa}, we introduce the set of $\g$-Diophantine vectors as

\begin{equation} \label{diofa}
\DD_\g
:=\bigg\{  \omega \in  \RRR^\ZZZ: \, | \omega\cdot\nu|\ge \g \prod_{j\in \ZZZ}(1+\jap{j}^{2} |\nu_j|^{2})^{-(q+1)}
\quad \forall\, \nu\in\ZZZ^\ZZZ_f\setminus\{0\}\bigg\} ,
\end{equation}
where $\g>0$.

%


Let us consider a  Hamiltonian of the form
\begin{equation}\label{ham}
	H:\matP\to\RRR,\qquad\qquad 
	H(\theta,I) = \om_0\cdot I + \frac{1}{2}  I \cdot Q[I]   + \e f(\theta),
\end{equation}
and assume the following.
\begin{itemize}
\item (finite energy) $\om_0\in\ell^\io$
\item (twist condition) $Q=Q^{T},Q^{-1}\in\MM(\ell^\io_q)\cap\MM(\ell^\io)$
\item (regularity) $f$ is $(r,4s)$-regular for any $r>0$
\item (smallness) $\e$ is sufficiently small w.r.t. the parameter $\g$
\end{itemize}

\begin{theo}[Main result of \cite{CGP}]\label{mainbc}
Fix $\g>0$ small enough.
	Under the hypotheses above there exists $\e_0>0$ such that for all $\e\in(-\e_0,\e_0)$ there exists a positive measure set $\mathcal C\subset B_{1/2}(\ell^\io_q)$ and two functions
	\[
	\al: \TTT^\ZZZ \times \mathcal C\times (-\e_0,\e_0) \to \ell^\infty \,,\qquad  A: \TTT^\ZZZ \times \mathcal C\times (-\e_0,\e_0) \to \ell^\infty_q\,,
	\]
	which are $2s$-analytic w.r.t.~$\f\in\TTT^\ZZZ$, such that
	\begin{equation}\label{toro}
		u:\TTT^\ZZZ \to \matT_\e(I_0),\qquad\qquad u(\f)=u(\f;I_0,\e)=(\f+\al(\f;I_0,\e),I_0 +A(\f;I_0,\e))
	\end{equation}
	parametrizes  an invariant torus $\matT_\e(I_0)\subset \matP$ on which the dynamics is linear with frequency $\om=\om(I_0):=\om_0 + Q[I_0]$, i.e.~setting
	\begin{equation}\label{dinamica}
		\Psi^t_\om : \TTT^\ZZZ \to \TTT^\ZZZ \qquad\qquad
		\f \mapsto \f + \om t
	\end{equation}
	one has
	\begin{equation}\label{soluzione}
		u \circ \Psi^t_\om = \Phi^t_H \circ u
	\end{equation}
	where $\Phi^t_H=\Phi_{X_H}^t$ is the hamiltonian flow. 
	The set $\CCCC$ is defined as 
\[
\CCCC:=\{I_0 \in B_{1/2}(\ell^\io_q) \; : \; \om(I_0) \in\DD_\g\}.
\]
Finally the torus map \eqref{toro} satisfies the bounds 
	\begin{equation}\label{stibound}
	 \| \al\|_{2s;\ell^\io} \le C \e,\qquad
\| A\|_{2s;\ell_q^\io} \le C \e .
\end{equation}

\end{theo}

%
Note that for such results to hold we do not
need $\Phi^t_H$ to be well defined on all $\matP$ but only on $\matT_\e(I_0)$, where it
is well defined by construction, and one has
\[
\left.\Phi^t_H \right|_{\matT_\e(I_0)} = u \circ \Psi^t_\om \circ u^{-1}.
\]

Note that the hypothesis $Q,Q^{-1}\in \MM(\ell^\io)$ is needed to obtain the measure estimates
on the set $\CCCC$; see Appendix \ref{misura}.
From now on the dependence on $I_0,\e$ will be omitted unless needed.

Our aim is to prove that the tori provided by \cite{CGP} are stable in the following sense.

\begin{theo}\label{stimadeitempivera}
For any $I_0 \in \CCCC$ there exists $\de_0$ small enough such that,
for all $\de<\de_0$ and
for any $(\theta_0,I_0)\in \matU_{\de/16}(\matT_\e(I_0))$, 
 the hamiltonian flow $\Phi_{H}^t$ is well defined and it remains in $\matU_{\de}(\matT_\e(I_0))$ 
for all
\[
|t| < C_1 (\e\de)^{-C_2\Big( {\log(\frac{1}{\e\de})}\Big)^{1/3} }
\]
for some appropriate constants $C_1,C_2$.
\end{theo}

Our proof is made of two steps.

First we prove that the tori provided by \cite{CGP} are
isotropic, and in fact lagrangian so that,
following an idea of Moser (see for instance \cite{BB} for a review),
we construct symplectic variables $(\f,y)$ for which the invariant torus corresponds to $y=0$.
In other words we prove the following result.

\begin{theo}\label{bibi}
For any $I_0\in\CCCC$ the torus $\matT_\e(I_0)$ is lagrangian. Moreover
there is analytic symplectic map $\Upsilon:\matP_{r}\to \matP_{2r}(I_0)$ for all $r>C\e$
for some $C$, such that
\[
H\circ \Upsilon (\f,y) = \om\cdot y + \frac{1}{2}y\cdot Q[y] + R(\f,y),\qquad R=O(\e\|y\|_{q}^2),\quad \|R\|_{r,s}<\e C,\quad \om=\om(I_0).
\]
\end{theo}

At the formal level, such a change of variables is the same as in the classical literature
on the topic. However in the present case the torus itself is infinite dimensional so
some care is needed. Precisely, one first has to define carefully $(r,s)$-analytic differential
forms and show that they behave as their finite dimensional counterparts, then
one has to verify that the change of variables not only is formally symplectic, but it maps
the phase space into itself, and the Hamiltonian in the new variables is still regular.

Second we use the functional setting introduced
in Section \ref{poisson} to show that one can perform $N$ steps of ``Birkhoff Normal Form'' as it
happens in the finite dimensional case, so as to obtain the following.

Set 
\begin{equation}\label{di}
D[y] :=  \om\cdot y + \frac{1}{2}y\cdot Q[y].
\end{equation}

\begin{theo}\label{birkof}
For any $I_0\in\CCCC$ and any finite $N>0$, there is a constant $C>0$
such that, fixing  
\begin{equation}
 \de_0 \le C\min\{ e^{-2N^3}\label{due1} ,\sqrt{\e}\}
\end{equation}
we have that, for any $\de\in(0,\de_0)$ there is an analytic symplectic change of variables
$\Upsilon_N: \matP_{\de/2}\to \matU_{\de}$ with
\[
\matU_\de=\matU_\de(\matT_\e(I_0)):= \{z=(\theta,I) \in \matP \,:\, \dist(z,\matT_{\e}(I_0))<\de\}
\]
 such that
\begin{equation}\label{norm}
H\circ\Upsilon_N (\f,y) = D[y] + {Z}_N(y) + {R}_N(\f,y),
\end{equation}
where $\om=\om(I_0)$, $Z_N,R_N$ are $(\de/2,s/2)$-regular. Moreover one has
\[
\|Z_N\|_{\delta/2,s/2} \le 2 \|R\|_{\de,s}
	\,,\quad 	\|R_N\|_{\delta/2,s/2} \le C^N \| R\|_{\de,s}^{N+1}
	e^{ \mathtt C' N^{7/4}},
\]
for some $\mathtt{C}'>0$.
\end{theo}

Finally, Theorem \ref{stimadeitempivera} follows by optimizing $N$ and usual Duhamel
estimates; see for instance \cite{BMP1}.

\medskip

Theorem \ref{mainbc} in fact applies to more general Hamitonians of the form
\[
H(\theta,I) = h(I) + \e f(\theta,I)
\]
provided that one imposes some sufficiently strong hypotheses regarding the analyticity w.r.t.~$I$
and some appropriate twist condition;
such a result is obtained in \cite{BC}, which is in preparation. Precisely one needs to require that
$\del^2_I h (I),(\del^2_I h (I))^{-1}, {\del^2_I f(\theta,I)}\in \MM(\ell^\io_q)\cap \MM({\ell^\io})$ for
all $I$ in some ball $B_\rho$.
Then Theorem \ref{mainbc} holds almost verbatim with $\om(I_0)=\del_I h(I_0)$ and 
$\CCCC\in B_\rho$; in particular the measure estimates
requires $(\del^2_I h (I))^{-1}\in \MM(\ell^\io)$ as it happens in the present case; see Appendix \ref{misura}.

Then Theorem \ref{stimadeitempivera} should be easily generalizable also in such a case.
In fact in our proof we did not
take full advantage of the very special form of our Hamiltonian \eqref{ham}, but
 we used it only to obtain the third bound in \eqref{staltribound}.
We believe however that an analogous bound can be obtained also in the general case,
so the rest of the proof would be identical.

\zerarcounters 
\section{Functional setting and bounds} 
\label{poisson} 

Here we provide some useful structure and bounds. In this section we start from very general 
properties on $(r,s)$-analytic functions and then specify to $(r,s)$-regular functions.

First of all we state the following result, which extends classical Cauchy estimates, proved in
\cite{CMP} in the case action-independent case;
see also \cite{lore} for a general proof in the case $|\cdot|_\star = \|\cdot \|_{\ell^1}$,
which is in fact more technically challenging.

\begin{lemma}\label{cau}
If a function $\matF:\matP_r\to \ttV$ is $(r,s)$-analytic, then its differential $d\matF$
is an $(r',s')$-analytic map with values in $\MM(\TT,\ttV)$ for any $r'<r$, $s'<s$, with quantitative bounds
on the operator norm given by the Cauchy estimates
\begin{equation}\label{cauchy}
\|d\matF\|_{r',s',\MM(\TT,\ttV)} \le \frac{C}{\min\{r-r',s-s'\}} \|\matF\|_{r,s,\ttV}
\end{equation}
where $C=C(q)$ is a constant. In particular if $(\ttV,\|\cdot\|_\ttV)=(\RRR,|\cdot|)$ then $\MM=T^*\matP$.

\end{lemma}

Similarly
one has the following result, proved in \cite{lore} in the case $|\cdot|_\star = \|\cdot \|_{\ell^1}$,
which extends classical results on flows and pull-backs in finite dimension. The proof in our
case can be obtained essentially following \cite{lore} word by word, and it is in fact easier.

\begin{lemma}\label{flussi}
Given a $(r,s)$-analytic vector field $X$, for any $\rho\in(0,r/2)$ and any $\s\in(0,s/2)$
 there exist $T_0>0$ and an analytic diffeomorphism  
\begin{equation}\label{riflussi}
\Phi_X^t : \matP_{r-\rho} \to \matP_{r}
\end{equation}
analytic in $t\in (-T_0,T_0))$ such that for any $\matF:\matP_{r}\to \ttV$ $(r,s)$-analytic
\ and any $t\in(-T_0,T_0)$ one has that
$(\Phi_X^t)^* \matF$ is $(r',s')$-analytivc
with bounds 
\begin{equation}\label{stobound}
\|(\Phi^t_X)^* \matF - \matF\|_{r-\rho,s-\s,\ttV} \le \frac{C |t|}{\min\{\rho/r,\s/s\}} \|X\|_{r,s,T\matP}\|\matF\|_{r,s,\ttV}.
\end{equation}
\end{lemma}

\begin{lemma}\label{sreg}
If $H$ depends on the angles and is $s$-analytic then it is $(r,s-\s)$-regular for any $r>0$
and any $\s\in(0,s)$, with the bound
\[
\| H\|_{r,s-\s} \le \frac{C}{r\s^{2q}}\|H\|_{s,\RRR},
\]
for some constant $C=C(q)$.
\end{lemma}

\prova
By definition $X_H=(0,-\del_\theta H)$ and hence
\[
\begin{aligned}
\| H\|_{r,s-\s} &= \|X_{H}\|_{r,s-\s,T\matP} = \frac{1}{r} \|\del_\theta H\|_{s-\s,\ell^{\io}_q} \\
&\le
\frac{1}{r} \sup_{j\in\ZZZ} \jap{j}^q \sum_{\nu\in\ZZZ^\ZZZ_f} |\nu_j | |H_\nu| e^{(s-\s)|\nu|_\star}\\
&\le  \frac{1}{r} \sup_{j\in\ZZZ} \jap{j}^q \sum_{\substack{\nu\in\ZZZ^\ZZZ_f  \\ \nu_j\ne0}}
	|\nu_j | |H_\nu| e^{(s-\s)|\nu|_\star}\\
& \le\frac{1}{r} \|H\|_{s,\RRR}\sup_{j\in\ZZZ}\sup_{\substack{\nu\in\ZZZ^\ZZZ_f  \\ \nu_j\ne0}}
	\jap{j}^q|\nu_j | e^{-\s|\nu|_\star} 
\end{aligned}
\]
so that the bound follows by recalling that $|\nu|_\star> \sqrt{\jap{j}}|\nu_j|$ since $\nu_j\ne0$.
\EP

Given $K\in\NNN$ and $H\in\matH_{r,s}$ we define the ultraviolet cutoff as
\begin{equation}\label{uv}
\Pi_K H(\theta,I) := \sum_{|\nu|_\star \le K} H_{\nu}(I)e^{i\theta\cdot\nu},
\qquad\qquad
\Pi_K^\perp := (\uno - \Pi_K)
\end{equation}

We have the following result.

\begin{lemma}\label{genoveffa}
	If $H\in \matH_{r,s}$, then for all $s'\le s$  and $r'\le r$ one has
	\begin{equation*}
		\norm{H}_{r',s'} \le \max\{\frac{s}{s'},\frac{r}{r'}\} \norm{H}_{r,s}\,.
	\end{equation*}
	If moreover $H\in \matH_{r,s}$ has minimal degree $d\ge 1$ in $y$ then for all $r'\le r$ one has
	\begin{equation*}
		\norm{H}_{r',s} \le \pa{\frac{{r'}}{r}}^{d-1} \norm{H}_{r,s}\,.
	\end{equation*}
	Finally for any $s'\le s$ and any $K\in\NNN$ one has
	\[
	\|\Pi^{\perp}_K H \|_{r,s'} \le e^{-(s-s')K}\big( \frac{s}{s'}\big) \| H\|_{r,s}.
	\]
\end{lemma}

\prova
The bounds follow as a direct consequence of the Definitions.
We prove only the last one. In fact
\[
\begin{aligned}
\|\Pi^{\perp}_K H \|_{r,s'} &= \big\|\sum_{|\nu|_\star > K} H_{\nu}(I)e^{i\theta\cdot\nu}\big\|_{r,s'} \\
&=  \sum_{k\in\NNN^\ZZZ_f}\sum_{\substack{\nu\in\ZZZ^\ZZZ_f \\ |\nu|_\star > K}} 
\|(X_H)_{k,\nu}\|_{r,s',T\matP} e^{s'|\nu|_\star}(r w)^k\\
&\le  e^{-(s-s')K}\big( \frac{s}{s'}\big) \sum_{k\in\NNN^\ZZZ_f}
\sum_{\substack{\nu\in\ZZZ^\ZZZ_f \\ |\nu|_\star > K}} \|(X_H)_{k,\nu}\|_{r,s,T\matP} e^{s|\nu|_\star}(r w)^k,
\end{aligned}
\]
so the assertion follows.
\EP

Given any $H,K\in\matH_{r,s}$, the symplectic $2$-form naturally induces the Poisson brackets
\[
\{H,K \} := \Omega[X_H,X_K],
\]
so that, as a consequence of the Cauchy estimates in Lemma \ref{cau} we have the following result.

\begin{lemma}[Poisson brackets and Hamiltonian flow]\label{hamflow}
For any $\rho\in(0,r/2)$, and $\s\in(0,s/2)$, and any $H,K\in\matH_{r,s}$
one has that  $\{H,K\}$ is $(r-\rho,s-\s)$-regular  with the bound
	\begin{equation}\label{mojoodo}
	\|\{H,K\}\|_{r-\rho,s-\s} \le \frac{C}{\delta}\|H\|_{r,s}\|K\|_{r,s}\,, \qquad \delta= \min\{ \frac{\rho}{r},\frac{\s}{s}\}.
	\end{equation}
Moreover, for any $S\in\matH_{r,s}$ with 
	\begin{equation}\label{stima generatrice}
		\norm{S}_{r,s} \leq \frac{\delta }{2 C  e} \,
	\end{equation} 
 the time $1$-Hamiltonian flow 	
	$\Phi^1_S: \matP_{r-\rho}\to
	\matP_{r}$  is well defined, analytic and symplectic.
	Finally, for any $H\in \matH_{r,s}$
	we have that
	$H\circ\Phi^1_S= e^{\set{S,\cdot}} H\in\matH_{r-\rho,s-\s}$ with
	\begin{equation}
		\label{tizio}
		\norm{e^{\{S,\cdot\}} H}_{r-\rho,s-\s}  \le 2 \norm{H}_{r,s}\,,
	\end{equation}
	and more generally for any $h\in\NNN$ and any sequence  $\{c_k\}_{k\in\NNN}$ with $| c_k|\leq 1/k!$, we have 
	\begin{equation}\label{brubeck}
		\norm{\sum_{k\geq h} c_k {\rm ad}^k_S\pa{H}}_{r-\rho,s-\s} \le 
		2 \|H\|_{r,s} \big(\frac{eC}{\delta}\|S\|_{r,s}\big)^h
		\,,
	\end{equation}
	where  ${\rm ad}_S\pa{\cdot}:= \set{S,\cdot}$.
\end{lemma}

\begin{proof}
	Recalling that $X_{\{H,K\}}=[X_H,X_K]= dX_H[X_K] - dX_K[X_H]$,
	the Cauchy estimates \eqref{cauchy} hold with $\ttV= T\matP$, so we deduce \eqref{mojoodo}.
	The other bounds are rather standard of this fact; see for instance \cite{BMP}.
\end{proof}

As a consequence of Theorem \ref{mainbc} we deduce the following bounds.

\begin{lemma}\label{altribound}
Under the hypotheses of Theorem \ref{mainbc} one has
	\begin{equation}\label{staltribound}
	\begin{aligned}
& \|  d \al ^T\|_{ 3s/2; \mathcal M(\ell^\io_q)} \le C\e\,,\qquad
 \|  d \al \|_{ 3s/2; \mathcal M(\ell^\io)} \le C\e\,\\
& \|  d A ^T\|_{ 3s/2; \mathcal M(\ell^\io,\ell^\io_q)} \le C\e,\qquad
\| \| \del_\f {d} \al[\cdot]\|_{3s/2;\ell^\io_q} \|_{\MM(\ell^\io)} < C\e.
 \end{aligned}
\end{equation}
\end{lemma}

\prova
We only prove the first bound in \eqref{staltribound}: the second follows by Cauchy estimates,
the third follows from the fact that $A= Q^{-1}[\om\cdot \del_\f \al]$ together with the first bound,
while the last is again a consequence of the Cauchy estimates and the first bound.
By definition we need to prove that
\[
\sup_{j\in\ZZZ}\jap{j}^{q}\sum_{i\in\ZZZ} \|\del_{\f_j} \al_i\|_{3s/2,\RRR} \frac{1}{\jap{i}^q} < C\e.
\]
Expanding explicitly and using Lemma \ref{sreg} we obtain
\[
\begin{aligned}
\sup_{j\in\ZZZ}\jap{j}^{q}\sum_{i\in\ZZZ} \|\del_{\f_j} \al_i\|_{3s/2,\RRR} \frac{1}{\jap{i}^q} &\le
\sup_{j,i\in\ZZZ}\jap{j}^{q} \|\del_{\f_j} \al_i\|_{3s/2,\RRR}\sum_{i\in\ZZZ} \frac{1}{\jap{i}^q} \\
&  \le C \sup_{i\in\ZZZ}\|\del_{\f} \al_i\|_{3s/2,\ell^\io_q} \le \frac{C}{s^{2q}}\|\al\|_{2s,\ell^\io}
\end{aligned}
\]
so the assertion follows.
\EP

\zerarcounters 
\section{The isotropic torus} 
\label{toroiso} 

Recall that $u$ is the torus embedding in \eqref{toro}.

\begin{lemma}\label{isotropo}
The torus $\matT$ is isotropic in $\matP$, i.e.
\begin{equation}\label{iso}
u^*\Omega = 0.
\end{equation}
\end{lemma}

\prova
First observe that we may write in coordinates
\[
u^*\Omega(\f) = \sum_{\substack{i,j\in\ZZZ \\ i < j}} (M_u(\f))_{i,j} d\f_i\wedge d\f_j
\]
and
\[
(u \circ \Psi^t_\om)^*\Omega (\f) =  (\Psi_\om^t)^*(u^*\Omega) = 
 \sum_{\substack{i,j\in\ZZZ \\ i < j}} (M_u(\f + \om t) )_{i,j}d \f_i \wedge d \f_j.
\]
We claim that $M_u\equiv0$. In fact, on the one hand one has
\[
L_\om[u^*\Omega] =  \left.\frac{d}{dt}\big( (\Psi_\om^t)^*(u^*\Omega) \big)\right|_{t=0} =
 \sum_{\substack{i,j\in\ZZZ \\ i < j}}\om\cdot\del_\f  (M_u(\f ))_{i,j} d \f_i \wedge d \f_j.
\]
On the other hand, since $d(u^*\Omega)=u^*d\Omega=0$, by Cartan formula \eqref{carnanazzo} we obtain
\begin{equation}\label{eccodoveserve}
L_\om(u^*\Omega) = d((u^*\Omega)[\om,\cdot]).
\end{equation}
Now, since $u$ is an invariant torus one has 
\begin{equation}\label{diffv}
du[\om]=\om\cdot\del_\f u = X_H(u) 
\end{equation}
and hence, using also \eqref{unicop}, we obtain
\begin{equation}\label{contazzo}
\begin{aligned}
u^*\Omega[\om,\cdot] &= \Omega[du[\om],du[\cdot]] = \Omega[ X_H(u), du[\cdot]] \\
&= dH(u)[du[\cdot]] = d(u^*H)[\cdot],
\end{aligned}
\end{equation}
so that clearly
\begin{equation}\label{contazzo2}
L_\om (u^*\Omega) = d(d(u^*H)) = 0.
\end{equation}
which means that $M_u$ is constant on $\TTT^\ZZZ$. Finally $M_u$ has zero average because
$u^*\Omega = d(u^* \lambda)$ with
\begin{equation}\label{pull}
(u^* \la)(\f) = \sum u_k(\f) d\f_k,
\end{equation}
so that
\[
(M_u)_{i,j}= \del_{\f_i}u_j - \del_{\f_j }u_i
\]
and the assertion follows.
\EP

Fix $I_0\in\CCCC$ and set
\begin{equation}\label{primo}
\Upsilon(\f,y)=(\Upsilon_\theta(\f,y),\Upsilon_I(\f,y)):=( \f + \al(\f) ,I_0 + A(\f) + (\uno+\mathtt d \al(\f))^{-T}[ y])
\end{equation}
where $\al,A$ are those given in \eqref{toro}.

In what follows, we shall systematically make use of the following facts.

\begin{enumerate}

\item If $f$ is a real valued  $(r,s)$-analytic function and $\matF$ is $(r,s)$-analytic function 
with values in $\ttV$, then $f\matF$  is $(r,s)$-analytic with values in  $\ttV$ and
$\| f\matF\|_{r,s;\ttV} \le C \|f\|_{r,s}\|\matF\|_{r,s;\ttV} $.

\item If $\Phi$ is  $(r,s)$-analytic with values in $\ell^\io_q$  with $\|\Phi\|_{r,s;\ell^\io_q}< R$ and 
$\matF$ is  an $(R,s)$-analytic function  with values in $\ttV$, then 
$\matG(\f,y):= \matF(\f,\Phi(\f,y))$ is  $(r,s)$ analytic with values in $\ttV$ and satisfying
\[
\|\matG\|_{r,s;\ttV} \le C \|\matF\|_{R,s;\ttV} 
\]

\item If $\| \alpha\|_{r,s;\ell^\io}< \sigma \ll 1$ and 
$\matF$ is  an $(r,s-\s)$-analytic function  with values in $\ttV$, then 
 then 
$\matG(\f,y):= \matF(\f+\al(\f),y)$ is  $(r,s)$ analytic with values in $\ttV$ and satisfying
\[
\|\matG\|_{r,s;\ttV} \le C \|\matF\|_{r,s-\s;\ttV} 
\]

\end{enumerate}

Note that while 1 and 2 are essentially a direct consequence of the definition (see for instance
\cite{lore}), 3 is less trivial, but it is proved for instance in \cite{MP,CMP}.

\begin{lemma}\label{sympatica}
For any $r>C\e$ and any hamiltonian $H$ which is $(2r,2s)$-regular at $I_0$, one has
that $\Upsilon^*H$
is $(r,s)$-regular. Moreover $\Upsilon$ is symplectic.
\end{lemma}

\prova
Let us start by proving that $\Upsilon$ maps $\matP_{r}(0)$ in  $\matP_{2r}(I_0)$.

By item $3$ the $\theta$-component $\Upsilon_\theta$ is a diffeomorphism on  $\TTT^\ZZZ$.
Regarding $\Upsilon_I$, we use the first inequality in \eqref{staltribound} and the second 
in \eqref{stibound} to deduce that
 by Neumann series, $\|(\uno+\mathtt d \al(\theta))^{-T} [y]\|_{\ell^\io_q} \le (1+C\e)\|y\|_{\ell^\io_q} $
and the result follows provided that $\e$ is sufficiently small so that
\[ 
\| A\|_{3s/2;\ell_q^\io}  + (1+C\e)r\le   r + C'\e(1+r) < 2r\,. 
\]

It remains to show that, if $H_{I_0}$ is $(2r,2s)$-regular, setting $G:= H\circ \Upsilon$,
then $G$ is $(r,s)$-regular.
Let us start by computing the Hamiltonian vector field of $G$
\[
\begin{aligned}
	\del_{y} G &= (\uno + d \al)^{-1} [(\del_I H) \circ \Upsilon] \,,\\
	\del_{\f} G &= (\uno + d \al)^{T}[ (\del_\theta H) \circ \Upsilon] + 
		( d A)^{T}[ (\del_I H )\circ \Upsilon] \\
		&\qquad+\Big( \del_{\f} ( d \al)(\uno + d \al)^{-1} [(\del_I H) \circ \Upsilon]\Big)\cdot   
		\Big((\uno + d \al)^{-T} [y]\Big).
\end{aligned}
\]
Then we need to prove that $\del_y G, \del_{\f} G$ are $(r,s)$-analytic,
with values in $\ell^\io$ and $\ell^\io_q$ respectively. 

The second in \eqref{staltribound} implies that
 the operator $(\uno+d\al)^{-1}$ maps the set of $(r,s)$-analytic functions with values in $\ell^\io$
into itself. Now since $\del_I H$ is $(2r,2s)$-analytic map with values in $\ell^\io$,
then by items 2 and 3 $\del_I H \circ \Upsilon$ is $(r,s)$-analytic with values in $\ell^\io$, and hence
$\del_y G$ is $(r,s)$-analyitc.

To analyse $\del_\f G$ we start noticing that for the first summand one can reason as above switching
the roles of $\ell^\io$ and $\ell^{\io}_q$. For the second summand we use the
third inequality in \eqref{staltribound}
and then reason as above. Regarding the third summand, again we reason as above,
using the last bound
in \eqref{staltribound}.

To conclude
we need to prove that $\Upsilon^*\Omega=\Omega$. To this purpose it is sufficient to show that
$\Upsilon^* \la - \la$ is closed, since in such a case one has
\[
\Upsilon^*\Omega = \Upsilon^*d\la =d \Upsilon^*\la = d\la =\Omega.
\]
By direct inspection 
\[
\begin{aligned}
\Upsilon^* \la - \la &= (I_0 + A(\f) + (\uno+\mathtt d \al(\f))^{-T} y)d( \f + \al(\f)) - yd\f \\
&=  (I_0 + A(\f) )d( \f + \al(\f)) = u^*\la
\end{aligned}
\]
which is closed since $d(u^*\la) = u^*\Omega=0$ by Lemma \ref{isotropo}.
\EP

\noindent
{\it Proof of Theorem \ref{bibi}.}
We apply Lemma \ref{sympatica} to the Hamiltonian \eqref{ham}, which is $(2r,2s)$-regular
for any $r$ by hypothesis. By construction the invariant torus $\matT_\e(I_0)$ in the variables $(\f,y)$
corresponds to $y=0$ and it is a Kronecker torus with frequency $\om(I_0)$, so that
\[
R(\f,y) := H\circ \Upsilon(\f,y) - \om(I_0)\cdot y - \frac{1}{2}  y \cdot Q[y]
\]
 is quadratic w.r.t. $y$, and the bound follows
directly. 
In fact, substituting  we obtain that
\[
R(\f,y)= \frac12 (\uno +d\al)^{-T} [y] \cdot Q(\uno +d\al)^{-T} [y]  - \frac{1}{2}  y \cdot Q[y] \,.
\]
Finally, by direct inspection the torus $\TTT^\ZZZ\times\{0\}$ is lagrangian, namely its tangent
at each point coincides with its symplectic orthogonal: since $\Upsilon$ is symplectic this ensures
that also $\TT_\e(I_0)$ is lagrangian.
\EP 

\begin{rmk}
	It is worth mentioning that not only $R$ is $(r,s)$-regular for any $r$, but it also satisfies the following 
	\[
	\| \del_y R\|_{r,s,\ell^\io_q} \le \e C\,,\qquad  	\|\del_y^2R\|_{r,s,\mathcal M(\ell^\io)} \le \e C
	\]
	Indeed the first bound comes from  the first inequality in \eqref{staltribound}, and the fact that,
	$\al -\langle \al\rangle =(\omega\cdot \del_\f)^{-1} Q A$ 
	has in fact values in $\ell^\io_q$ by Lemma 4.6 in \cite{CGP}. 
Regarding the second bound we first note that
\[
\del^2_y R = (\uno+d\al)^{-1}Q(\uno+d\al)^{-T} - Q.
\]
Again the second inequality in \eqref{staltribound} ensures that $(\uno+d\al)^{-1} \in \MM(\ell^\io)$,
while Lemma 4.6 in \cite{CGP} ensures that $\al - \av{\al} \in \ell^\io_q$ so that
\[
\sup_{j\in\ZZZ}\sum_{i\in\ZZZ} \|\del_{\f_j}\al_i\|_{s,\RRR} < \e C,
\]
and hence $d\al^T\in\MM(\ell^\io)$.
\end{rmk}

\section{Birkhoff Normal Form}

By construction in the variables $(\f,y)$ the Hamiltonian has the form
\begin{equation}\label{newham}
H_0(y,\f)= D[y] + R(y,\f)\,,\quad R(y,\f):= \sum_{\substack{k\in\NNN^\ZZZ_f\\ \|k\|_{\ell^1}\ge 2}}\sum_{\nu\in\ZZZ^\ZZZ_f} R_{k,\nu} y^ke^{\ii \f\cdot \nu}
\end{equation}
where $D=D[y]$ is defined in \eqref{di} with $\om=\om(I_0)$ diophantine,  and $\| R\|_{r,s}\le Cr\e$. In particular, since
$R$ is quadratic w.r.t.~$y$, then
\begin{equation}\label{numeropuro}
\frac{d}{d\de} \frac{\| R\|_{\de,s}}{\de} = 0,\qquad  \frac{\| R\|_{\de,s}}{\de}\le C\e.
\end{equation}

We now  perform $N$ averaging steps namely construct a change of variables, defined in some small neighborhood of the invariant torus $\TTT^\ZZZ\times \{0\}$, which conjugates $H_0$ to the normal form
\[
H_N(y,\f)= D[y] + Z_N(y)+R_{N}(y,\f)\,,\qquad
R_N(y,\f):=\!\!\!\! \sum_{\substack{k\in\NNN^\ZZZ_f\\ \|k\|_{\ell^1}\ge N+ 2}}\sum_{\nu\in\ZZZ^\ZZZ_f}
R^{(N)}_{k,\nu} y^ke^{\ii \f\cdot \nu}
\]
where $R_N$ is regular while $Z_N$ is integrable, regular,  and of degree at least two in $y$.

We have the following standard result.

\begin{lemma}\label{homo}
	Assume that $\omega$ is Diophantine and fix $K\in\NNN$. There exists $r(K)$
	 such that, if $0<r<{r(K)}$ 
	 the following holds for any $s>0$. If $f$ is $(r,s)$-regular and has zero average then the equation
	\begin{equation}\label{homo}
	\{D, S\} =\Pi_K f\,
	\end{equation}
	admits a unique solution $S = \Pi_K S$ with zero average satisfying the bounds
	\begin{equation}\label{stimomo}
	\|S\|_{r,s} \le C_0 e^{\mathtt C K^{2/3}\log(K)}\|f\|_{r,s}
	\end{equation}
	for some constants $C_0$, $\mathtt C$.
\end{lemma}

\prova
We start by noting that \eqref{homo} can be rewritten as
\[
(\om + Q[y])\cdot \del_\f S(y,\f) = \sum_{\substack{\nu\in\ZZZ^\ZZZ_f \\ 0< |\nu|_\star \le K}} f_\nu(y)e^{\ii\nu\cdot\f}
\]
so, passing to the Fourier side we obtain formally
\[
S_\nu(y) = (\ii(\om + Q[y])\cdot\nu)^{-1} f_\nu(y).
\]

The bound (iii) in \cite{CMP}-Lemma A.6 ensures
\[
|\om\cdot \nu| \ge C_1e^{C_2 K^{2/3} \log(K)}
\]
for some constants $C_1,C_2$. On the other hand if we set
\begin{equation}\label{erredikappa}
r(K) := \frac{1}{4C_1\|Q\|_{\MM(\ell^\io_q)} }e^{-2C_2K^{2/3} \log(K)},
\end{equation}
we can bound
\[
|(\om + Q[y])\cdot\nu | \ge \frac{C_1}{2} e^{C_2 K^{2/3}\log(K)}.
\]

Then we have $\del_\f S_\nu =\ii\nu S_\nu$ so is easily bounded,
while
\[
(\del_yS)_\nu(y) = (\ii(\om + Q[y])\cdot\nu)^{-1}  \del_y f_\nu(y) -
(\ii(\om + Q[y])\cdot\nu)^{-2} f_\nu(y) Q[\nu]
\]
so the bound \eqref{stimomo} follows with $\mathtt C = 3 C_2$.
\EP

We shall iterate the following ``Birkhoff step'':
consider $\calH (y,\f)= D[y] + \mathcal Z(y)+\mathcal R(y,\f)$, fix $K>1$ and assume that

\begin{itemize}
	\item  $\mathcal Z,\mathcal R$ are $(\delta,s)$-regular, with $\de < r(K)$ defined in \eqref{erredikappa}
	\item $\mathcal Z$ has minimal degree $2$ in $y$, while  $\mathcal R$ has minimal degree
	$d\ge2$ in $y$,
	\item there exist $0<\rho<\delta/2$ and  $0<\s<s/2$ such that
	\begin{equation}\label{bistep}
	\|\mathcal R\|_{\delta,s} e^{-\mathtt C K^{2/3}\log(K)} < c\min(\frac{\rho}{\delta}, \frac{\s}{s}),
	\end{equation}
	for some small constant $c>0$.
\end{itemize}

Let $\mathcal S$ be the unique solution of
\[
\{D, \mathcal S\} = \Pi_K\mathcal R-\langle \mathcal R\rangle\,.
\]
By Lemma \ref{homo} $\mathcal S$ is $(\delta,s)$-regular  and by Lemma  \ref{hamflow}
the time-one flow of $\mathcal S$ generates a symplectic change of variables which
maps $(\delta,s)$-regular functions in $(\delta-\rho,s-\s)$ regular ones.

We have
\[
\begin{aligned}
	e^{\{\mathcal S,\cdot\}} (D & + \mathcal Z+\mathcal R ) \\
	&=D + \mathcal Z + \Pi_K^\perp \mathcal R + \langle \mathcal R \rangle + 
	\sum_{k=1}^\io \frac{1}{k!}({\rm ad}\mathcal S)^k(\mathcal Z+\mathcal R) -
	\sum_{k=2}^\io \frac{1}{k!}({\rm ad}\mathcal S)^{k-1} (\Pi_K\mathcal R- \langle \mathcal R \rangle  )\\
	& = D  + \mathcal Z_++\mathcal R_+ 
\end{aligned}
\]
where $\mathcal Z_+:= \mathcal Z+\langle \mathcal R \rangle $  is integrable,
regular and of minimal degree $2$, while 
$$
\mathcal R_+:= \Pi_K^\perp \mathcal R +  \sum_{k=1}^\io \frac{1}{k!}({\rm ad}\mathcal S)^k(\mathcal Z+\mathcal R) -
\sum_{k=2}^\io \frac{1}{k!}({\rm ad}\mathcal S)^{k-1} (\Pi_K\mathcal R- \langle \mathcal R \rangle  )
$$ 
is regular and has minimal degree $d+1$. 
By Lemma \ref{genoveffa} and Lemma \ref{hamflow} we have the following bounds
\[
\|\mathcal Z_+\|_{\de-\rho,s-\s} \le
(\|\mathcal Z\|_{\delta,s} +
\|\mathcal R\|_{\delta,s}  ),
\]
\[
\|\mathcal R_+\|_{\de-\rho,s-\s}  \le e^{-\s K}\|\RR\|_{\de,s}+
C \max(\frac{\de}{\rho},\frac{s}{\s})e^{\mathtt C K^{2/3}\log(K)}
(\|\mathcal Z\|_{\delta,s}\|\mathcal R\|_{\delta,s}  + \|\mathcal R\|_{\delta,s}^2 ).
\]


\begin{lemma}[Iterative Lemma]\label{itero}
	Let $N,K\ge 2$ be large enough.  Fix  $s_0>0$ and $\delta_0 \in(0, r(K))$, and assume
	\begin{equation}\label{piccino}
	 e^{-s_0 K/(2N)} \ll
	\| R\|_{\delta_0,s_0}   \ll \frac{1}{N} e^{-3\mathtt C K^{2/3}\log(K)}
	\end{equation}
	Set $\rho= \delta_0/(2N)$, $\s= s_0/(2N)$ and, for $k=1,\dots,N$,
	set  $\delta_k= \delta_0- k \rho$, $s_k= s_0-k\s$.
	For any $k=1,\dots,N$  there exist a $(\delta_k,s_k)$-regular Hamiltonian
	\[
	H_k:= D + Z_k + R_k
	\]
	with $Z_k$ integrable and of minimal
	degree $2$ in $y$, and $R_k$ of minimal degree $k+2$ in $y$ such that the following holds.
	
\begin{enumerate}

\item Setting $Z_0=0$, $R_0=R$ and $S_k$ the unique zero-average solution of
\[
\{D,S_k\}= \Pi_{K} R_{k-1}-\langle R_{k-1}\rangle,
\]
each $S_k$ generates a time-one flow map $\matP_{\delta_k}\to \matP_{\delta_{k-1}}$
and one has $H_k =e^{\{S_k,\cdot\}} H_{k-1}$.

\item One has
\[
\|Z_k\|_{\delta_k,s_k} \le  2\sum_{h=0}^k 2^{-h} \|R_0\|_{\delta_0,s_0}
	\,,\qquad 	\|R_k\|_{\delta_k,s_k} \le C^k \| R_0\|_{\delta_0,s_0}^{k+1} ( {e^{3\mathtt C K^{2/3}\log(K)} N})^k.
\]

\end{enumerate}	
\end{lemma}

\prova
We prove the result by induction on $k$.
For $k=0$ the only thing to verify are the bounds in item 2, which hold trivially.

Let us assume inductively the result to hold for $k=0,\ldots,n-1$ with $n\ge 1$. 
First note that
\[
\norm{R_{n-1}}_{\delta_{n-1}, s_{n-1}} {e^{3\mathtt C K^{2/3}\log(K)} }N
\le  C^{n-1} \| R_0\|_{\delta_0,s_0}^{n}  ({e^{3\mathtt C K^{2/3}\log(K)} N})^n \le \frac{c}{2},
\]
provided $C$ is large enough,
so we can apply one ``Birkhoff step'' to $H_{n-1}$ and we obtain 
	\[
	\begin{aligned}
		\|Z_{n}\|_{\delta_n,s_n} & \le \|Z_{n-1}\|_{\delta_{n-1},s_{n-1}} + \|R_{n-1}\|_{\delta_{n-1},s_{n-1}} \\
		& \le  2\sum_{h=0}^{n-1}2^{-h}\|R_0\|_{\delta_0,s_0}  +  C^{n-1}  \| R_0\|_{\delta_0,s_0}^{n}
		({e^{3\mathtt C K^{2/3}\log(K)} N})^{n-1}
		\le 2\sum_{h=0}^{n}2^{-h}\norm{R_0}_{\delta_0,s_0}
	\end{aligned} 
	\]	
	where in the last inequality we used that
	\[
	C^{n-1}  \| R_0\|_{\delta_0,s_0}^{n}
		({e^{3\mathtt C K^{2/3}\log(K)} N})^{n-1}<2^{-n+1}\|R_0\|_{\de_0,s_0},
	\]
	by the smallness condition \eqref{piccino}.
	Similarly
	\begin{equation}\label{unaltronome}
	\begin{aligned}
		\|R_{n}\|_{\delta_n,s_n} &\le c{e^{3\mathtt C K^{2/3}\log(K)} N} (\|Z_{n-1}\|_{\delta_{n-1},s_{n-1}}\| R_{n-1}\|_{\delta_{n-1},s_{n-1}}  + \|R_{n-1}\|_{\delta_{n-1},s_{n-1}}^2 ) \\
		&\qquad +e^{- s_0K/(2N)} \|R_{n-1}\|_{\delta_{n-1},s_{n-1}}
		\\
		&\le 4c{e^{3\mathtt C K^{2/3}\log(K)} N}\|R_0\|_{\de_0,s_0}C^{n-1}\|R_0\|^n
		({e^{3\mathtt C K^{2/3}\log(K)} N} )^{n-1}\\
		&\qquad + e^{- s_0K/(2N)}C^{n-1}\|R_0\|^n
		({e^{3\mathtt C K^{2/3}\log(K)} N} )^{n-1}.
	\end{aligned}
	\end{equation}
	
Now we note that
\[
4c{e^{3\mathtt C K^{2/3}\log(K)} N}\|R_0\|_{\de_0,s_0}C^{n-1}\|R_0\|^n
		({e^{3\mathtt C K^{2/3}\log(K)} N} )^{n-1} \le \frac{1}{2}C^{n}\|R_0\|^{n+1}
		({e^{3\mathtt C K^{2/3}\log(K)} N} )^{n}
\]
provided that $C$ is large enough. Moreover
\[
 e^{- s_0K/(2N)}C^{n-1}\|R_0\|^n
		({e^{3\mathtt C K^{2/3}\log(K))} N} )^{n-1} \le \frac{1}{2}C^{n}\|R_0\|^{n+1}
		({e^{3\mathtt C K^{2/3}\log(K)} N} )^{n}
\]
follows by \eqref{piccino}, so the assertion follows.
\EP

\begin{rmk}\label{incasina}
If $R_0\equiv0$ the condition \eqref{piccino} means that one should take $K=+\io$. However in such a case
Theorem \ref{birkof} is trivial with $\Upsilon_N=\Upsilon$, and Theorem \ref{stimadeitempivera} holds
for infinite time. 
\end{rmk}

In view of Remark \ref{incasina} above, we assume $R_0\ne0$ and set
\[
\|R_0\|_{\de_0,s_0} = e^{-a}.
\]

 Setting $\Psi_N:= \Phi_{S_1}^1\circ \dots \circ\Phi_{S_N}^1:\matP_{\delta_0/2}\to\matP_{\de_0}$
 one has that $\Psi_N$
conjugates $H_0$ with the normal form \eqref{norm}, provided that 
\eqref{piccino} holds and $\de_0 < r(K)$. We now prove that for $\e,\de_0$
sufficiently small, this can be achieved by choosing $N,K$ suitably. Recall \eqref{numeropuro}.

\begin{lemma}\label{finale?}
For $\e,\de_0$ small enough,
assuming that
\begin{equation}\label{vedosubito}
\de_0 < (\frac{\|R_0\|_{\de_0,s_0}}{\de_0})^{1/2},
\end{equation}
 fix
\begin{subequations}
\begin{align}
&K=\frac{4}{s_0}a N,\label{uno} \\
& N<a^{1/3}\label{due} \\
&N<\frac{1}{2} \log^{1/3}(\frac{1}{\de_0}) . \label{tre}
\end{align}
\end{subequations}
Then $\de_0<r(K)$ and  \eqref{piccino} holds.
\end{lemma}

\prova
First of all for \eqref{piccino} to make sense one needs
	\begin{equation}\label{piccino1}
	 {s_0 K/(3N)} \ge {4\mathtt C K^{2/3}\log(K)} + \log(N)
	\end{equation}
	which in turn is implied by
	\begin{equation}\label{facaldo}
	\frac{K^{1/3}}{\log(K)} \ge \frac{13}{s_0}\mathtt C N.
	\end{equation}
By substituting \eqref{uno} in \eqref{due} we obtain $N^4 < s_0 K /4$ which ensures \eqref{facaldo}
provided $K$ is sufficiently large. Thus we need to prove
	\begin{equation}\label{ripiccino}
	 {s_0 K/(3N)} \ge
	a \ge 4\mathtt C K^{2/3}\log(K) + \log(N).
	\end{equation}
The first bound follows directly from the choice \eqref{uno}. The second follows from \eqref{facaldo}.
Thus we are left to verify that 
\[
\log(\frac{1}{\de_0}) > 3C_2(\frac{4}{s_0}a N)^{2/3} \log(\frac{4}{s_0}a N)
\]
which is implied by
\begin{equation}\label{quellali}
\log(\frac{1}{\de_0}) \ge (aN)^{3/4} = N^{3/4}(\log(\frac{1}{\de_0}) + \log(\frac{\de_0}{\|R_0\|_{\de_0,s_0}}))^{3/4}.
\end{equation}
Using \eqref{vedosubito},
then \eqref{quellali} becomes
\[
\log(\frac{1}{\de_0}) \ge 2 N^{3/4}(\log(\frac{1}{\de_0}))^{3/4},
\]
i.e. \eqref{tre}.
\EP

\medskip

\noindent
{\it Proof of Theorem \ref{birkof}.} First note that $\Upsilon$ conjugates the Hamiltonian $H$ in
\eqref{ham} to the Hamiltonian $H_0$ in \eqref{newham}. The smallness condition \eqref{due1}
ensures that the hypotheses of Lemma \ref{finale?} are met and hence one can apply Lemma
\ref{itero}. Then Theorem \ref{birkof} follows with $\Upsilon_N = \Upsilon \circ \Psi_N$.
\EP

Now, for any $\e,\de>0$ sufficiently small, optimizing over $N$ (see for instance \cite{giorgilli}) we conjugate
$H_0$ in \eqref{newham} to the hamiltonian
\[
H_{fin}(\f,y) = h_{fin}(y) + R_{fin}(\f,y),\qquad \|R_{fin}\|_{\de/2,s/2} < e^{-c ({\log(1/\epsilon)})^{4/3}}
\qquad \epsilon := \| R\|_{\de,s},
\]
for some constant $c>0$.

Finally, setting $\om_{fin}(y) = \del_{y} h_{fin} (y)$, we have the following result.

\begin{lemma}\label{stimadeitempi}
For any $(\f_0,y_0) \in \matP_{\de/4}$ the hamiltonian flow $\Phi_{H_{fin}}^t$ is well defined and satisfies
\[
\begin{aligned}
&\| y(t) - y_0\|_{\ell^\io_q} < \de^2 \\
&\|\f(t) - \f_0 - \om_{fin}(y_0)t\|_{\ell^\io} < \de s
\end{aligned}
\]
for all $t< T(\de)$ with
\[
T(\de) := C e^{c ({\log(1/\epsilon)})^{4/3}}
\]
for some constan $C>0$.
\end{lemma}

\prova
We know that $(\f(t),y(t))$ are well defined and belong to $\matP_{\de/2}$ at least for short time.
Let $T$ be such that $(\f(t),y(t))\in\matP_{\de/2}$ for all $|t|<T$. By Duhamel formula we obtain
\[
\begin{aligned}
&\| y(T) - y_0\|_{\ell^\io_q} \le \int_0^T \| \del_\f {H_{fin}}(\f(t),y(t))\|_{\ell^\io_q}dt \le
T \frac{\de}{2} \|R_{fin}\|_{\de/2,s/2} \\
&\|\f(T) - \f_0 - \om_{fin}(y_0)t\|_{\ell^\io}  \le \int_0^T \| \del_y {R_{fin}}(\f(t),y(t))\|_{\ell^\io}dt 
\le T \frac{s}{2} \|R_{fin}\|_{\de/2,s/2}.
\end{aligned}
\]
so the assertion follows.
\EP

Note that, for fixed $\e$, one has $T(\de)\to+\io$ as $\de\to0^+$.
Finally, we have

\begin{lemma}\label{stimadeitempifin}
For any $(\theta_0,I_0)\in \NN_{\de/16}(\matT_\e(I_0))$
 the hamiltonian flow $\Phi_{H}^t$ is well defined remains in $\NN_{\de}(\matT_\e(I_0))$ 
for all $|t|< T(\de)$.
\end{lemma}

\prova
It follows from Lemma \ref{stimadeitempi} and the fact that the change of variables $(\theta,I)\mapsto (\f,y)$
is Lipschitz with uniform bounds.
\EP

To conclude the proof of Theorem \ref{stimadeitempivera} it suffices to notice that,
since $\epsilon = \| R\|_{\de,s} \le C\e\de$, then 
\[
T(\de) \ge  C_1 e^{C_2 ({\log(\frac{1}{\e\de})})^{4/3}}
\]
for appropriate $C_1,C_2$.

\appendix

\zerarcounters 
\section{Analytic differential forms} 
\label{forme} 

Given a sequence space $\ttV$ we denote $\LL^n(\ttV)$ the set of $n$-linear forms
\[
F : \ttV^n \to \RRR
\]
which can be represented as $n$-tensors, i.e.
\[
F[v^{(1)},\ldots,v^{(n)}]=\sum_{j_1,\ldots,j_n\in I_\ttV} F_{j_1,\ldots,j_n} v^{(1)}_{j_1}\cdot\ldots\cdot v^{(n)}_{j_n}.
\]
We denote $\Lambda^n(\ttV)\subset \LL^{n}(\ttV)$ the set of alternate $n$-linear
forms on $\ttV$.

We define the space of $(r,s)$-analytic $n$ differential forms on $\matP$ as
\begin{equation}\label{difforme}
\DD^n_{r,s}(\matP):=\{F:\matP_r\to \Lambda^n(T\matP)\,:\, p\mapsto F(p)\in \Lambda^n(\TT),\ \|F\|_{r,s,\Lambda^n}<\io\}
\end{equation}
where
\[
\Lambda^n:=\Lambda^n(\TT)=\Lambda^{n}(\ell^\io(\RRR)\times\ell^\io_q(\RRR)),
\]
is endowed with the majorant operator norm
\begin{equation}\label{normaforme}
\|F\|_{\Lambda^n}:= \sup_{\substack{v_1,\ldots,v_n \in \TT \\ \|v_i\|=1}} |\und{F}[v_1,\ldots,v_n]|,
\end{equation}
where $\und{F}$ is defined as in \eqref{maggiorante}
and the norm $\|F\|_{r,s,\Lambda^n}$ is as in \eqref{converge!}. 

Given a diffeomorphism $\Phi:\matP_{r'}\to\matP_r$ we define the
usual pull-back of a differential form $F$
as
\begin{equation}\label{pbd}
\Phi^* F[v_1,\ldots,v_n] := F\circ\Phi[d\Phi[v_1],\ldots, d\Phi[v_n] ],
\end{equation}
and bounds in the spirit of \eqref{stobound} can be obtained.

Then we introduce
the standard exterior derivative
\begin{equation}\label{diff}
d: \DD^n_{r,s}(\matP) \longrightarrow \DD^{n+1}_{r',s'}(\matP)
\end{equation}
with $r'<r$, $s'<s$ and the Cauchy estimates \eqref{cauchy} hold with $\ttV=\Lambda^n$.
Given any $F\in \DD^n_{r,s}(\matP)$ we may represent it component-wise as
\begin{equation}\label{espandoespando}
\begin{aligned}
F(\theta,I) &= \sum_{n_1 + n_2 = n}\sum_{\substack{ j_1<\ldots<j_{n_1} \\ i_1< \ldots <i_{n_2}}}
F^{(\vec{\imath},\vec{\jmath})}(\theta,I) dI_{j_1}\wedge \ldots \wedge dI_{j_{n_1}} \wedge d \theta_{i_1}\wedge\ldots \wedge d \theta_{i_{n_2}}\\
&=\sum_{\substack{k\in\NNN^\ZZZ_f \\ \nu\in\ZZZ^\ZZZ_f} }
\sum_{n_1 + n_2 = n}\sum_{\substack{ j_1<\ldots<j_{n_1} \\ i_1< \ldots <i_{n_2}}} F_{k,\nu}^{(\vec{\imath},\vec{\jmath})}
 I^ke^{\ii \theta\cdot \nu}dI_{j_1}\wedge \ldots \wedge dI_{j_{n_1}} \wedge d \theta_{i_1}\wedge\ldots \wedge d \theta_{i_{n_2}}
\end{aligned}
\end{equation}
where we denoted $(\vec{\imath},\vec{\jmath}) = (j_1,\ldots,j_{n_1} , i_1, \ldots ,i_{n_2})$,
With the above notation the $2$-form $\Omega$ defined in \eqref{simpl} is well defined and belongs
to $\DD^2_{r,s}(\matP)$ for any $r,s$. Moreover $\Omega$ is exact, being the differential of the action
$1$-form
\begin{equation}\label{azione}
\lambda  := \sum_{j\in\ZZZ}I_j d\theta_j.
\end{equation}

Given a $(r,s)$-analytic vector field $X$ we define the interior product
\begin{equation}\label{interno}
\begin{aligned}
i_X : \DD^n_{r,s}(\matP) &\longrightarrow \DD^{n-1}_{r,s}(\matP) \\
   F \  &\longmapsto F[X,\cdots]
   \end{aligned}
\end{equation}
and the Lie derivative 
\begin{equation}\label{derivatadilie}
\begin{aligned}
L_X: \DD^n_{r,s}(\matP) & \longrightarrow \DD^n_{r',s'}(\matP) \\
 F \ & \longmapsto \frac{d}{dt}\left. \Big( (\Phi^t_X)^* F \Big)\right|_{t=0}
\end{aligned}
\end{equation}
where $(\Phi^t_X)^* F$ is the pull-back of $F$ via $\Phi^t_X$.

\begin{rmk}\label{cartan}
The choice of the norm in \eqref{difforme} allows us to obtain the Cartan formula
\begin{equation}\label{carnanazzo}
L_X = i_X\circ d + d \circ i_X.
\end{equation}
\end{rmk}

\zerarcounters 
\section{Measure estimates} 
\label{misura} 

Recall that $\om(I)=\om_0 + Q[I]$.
By definition the set $\CCCC$ is $B_{1/2}(\ell^\io_q)\setminus \mathscr{R}(\gamma)$
where

$$
\mathscr{R}(\gamma)=\bigcup_{\nu\in\ZZZ^\ZZZ_f\setminus\{0\}}
\underbrace{\left\{I\in B_{1/2}(\ell^{\infty}_q)\,:\, |\om(I)\cdot\nu|
<\gamma\prod_{j\in\ZZZ}(1+\langle j\rangle^{2}|\nu_j|^{2})^{-(q+1)}\right\}}_{:=\mathscr{R}^{(\nu)}(\gamma)}.
$$

Since $\om$ is continuous w.r.t.~the product topology, then $I\to\om(I)$ is a 
$\mu$-measurable map. This implies that, for all $\nu\in\ZZZ^\ZZZ_f\setminus\{0\}$,  
$\mathscr{R}^{(\nu)}(\gamma)$ is a $\mu$-measurable set.
For simplicity of notation, let us set

$$
\Delta^{(\nu)}:=\prod_{j\in\ZZZ}(1+\langle j\rangle^{2}|\nu_j|^{2})^{-({q}+1)}.
$$

By construction, $I\in\mathscr{R}^{(\nu)}_\rho(\gamma)$ 
if and only if it holds that

\begin{equation}\label{epifania}
   -\gamma\Delta^{(\nu)}-\om_0\cdot\nu<\sum_{j\in\ZZZ} I_j({Q}
   \nu)_j<\gamma\Delta^{(\nu)}-\om_0\cdot\nu;
   \quad ({Q}\nu)_j=\sum_{k\in\ZZZ}{Q}_{jk}\nu_k.
\end{equation}

Note that the sum $({Q}\nu)_j$ is actually finite, since $\nu\in\ZZZ^\ZZZ_f$; 
for the same reason, $\om_0\cdot\nu$ is a finite sum, while the sum on $j$ is, 
in general, not finite. Let $j_0\in\ZZZ$ be such that $\|Q\nu\|_\io=|(Q\nu)_{j_0}|$. 
Such an index exists since $Q\nu\in\ell^\infty(\RRR)$, as $Q$ is a bounded 
operator, and depends on the choice of $\nu\in\ZZZ^\ZZZ_f$, i.e.~$j_0=j_0(\nu)$.

If we set, for brevity,

$$
\mathfrak{a}(I_{j\neq j_0}):=\frac{1}{\|{Q}\nu\|_\io}\left(-
\gamma\Delta_{\m_{1,2}}^{(\nu)}-\om_0\cdot\nu-\sum_{\substack{{j\in\ZZZ}\\j\neq j_0}}
 I_j({Q}\nu)_j\right)
$$
and

$$
\mathfrak{b}(I_{j\neq j_0}):=\frac{1}{\|{Q}\nu\|_\io}\left(
\gamma\Delta_{\m_{1,2}}^{(\nu)}-\om_0\cdot\nu-\sum_{\substack{{j\in\ZZZ}\\j\neq j_0}} 
I_j({Q}\nu)_j\right),
$$
we can rewrite \eqref{epifania} as

\begin{equation}\label{letale}
\mathfrak{a}(I_{j\neq j_0})<I_{j_0}<\mathfrak{b}(I_{j\neq j_0}).  
\end{equation}

Note that

\begin{equation}\label{yppy}
  	  \begin{aligned}
 		   \sup_{\nu\ne0}\frac{1}{\|{Q}\nu\|_\io}\leq\sup_{\nu\ne0}
		   \frac{\|\nu\|_\io}{\|{Q}\nu\|_\io}=\sup_{y\ne0}
		   \frac{\|({Q}^{-1})y\|_\io}{\|y\|_\io}=\|{{Q}}^{-1}\|_{\MM(\ell^\io)}.
	   \end{aligned}
\end{equation}

If $I_{j_0}$ satisfies \eqref{letale}, then $I_{j_0}$ belongs to the interval 
$\Omega(I_{j\neq j_0}):=(\mathfrak{a}(I_{j\neq j_0}),\mathfrak{b}(I_{j\neq j_0}))$ 
which has measure equals

\begin{equation}\label{misurino}
    \meas(\Omega(I_{j\neq j_0})):=
    \frac{|\mathfrak{b}(I_{j\neq j_0})-\mathfrak{a}(I_{j\neq j_0})|}
    {\langle j_0\rangle^{-q}}=\frac{2\gamma\Delta^{(\nu)}}
    {\|{Q}\nu\|}\langle j_0\rangle^q.
\end{equation}
where we defined $\de_\nu(\gamma):=
\gamma\Delta_{\m_{1,2}}^{(\nu)}/\rho\|\mathcal{Q}\nu\|$. Then
 $\mathscr{R}^{(\nu)}_{\rho}(\gamma)$ is the 
$j_0$-normal domain

$$
\begin{aligned}
\mathscr{N}_\rho^{(j_0)}&:=\left\{I\in B_{1/2}(\ell^{\infty}_q)
\,:\, I_{j_0}\in\Omega(I_{j\neq j_0})\right\}\\
&=\prod_{\substack{{j\in\ZZZ}\\{j\leq j_0-1}}}
[-\frac{1}{2}\langle j\rangle^{-q},\frac{1}{2}\langle j\rangle^{-q}]\times\Omega(I_{j\neq j_0})
\times \prod_{\substack{{j\in\ZZZ}\\{j\ge j_0+1}}}
[-\frac{1}{2}\langle j\rangle^{-q},\frac{1}{2}\langle j\rangle^{-q}]
\end{aligned}
$$
and the measure estimate follows as in \cite{BMP1,Bjfa}.


\end{document}